\documentclass[11pt]{amsart}
\usepackage[margin=1.15in]{geometry}
\usepackage{amsmath,amssymb,amsfonts,mathtools,hyperref}
\usepackage{url}

\newcommand{\A}{\mathbb A}
\newcommand{\C}{\mathbb C}

\newcommand{\Q}{\mathbb Q}
\newcommand{\R}{\mathbb R}
\newcommand{\Z}{\mathbb Z}

\newcommand{\PP}{\mathbb P}

\newcommand{\Jac}{\operatorname{Jac}}
\newcommand{\Vol}{\operatorname{Vol}}
\newcommand{\wt}{\operatorname{wt}}
\newcommand{\disc}{\operatorname{disc}}
\newcommand{\divisor}{\operatorname{div}}
\newcommand{\calR}{\mathcal R}
\newcommand{\calW}{\mathcal W}
\newcommand{\calG}{\mathcal G}
\newcommand{\calM}{\mathcal M}

\newtheorem{theorem}{Theorem}[section]
\newtheorem{proposition}[theorem]{Proposition}
\newtheorem{lemma}[theorem]{Lemma}

\theoremstyle{definition}

\newtheorem{remark}[theorem]{Remark}

\numberwithin{equation}{section}

\title[Marked rational $3$-torsion in genus $2$]{Counting odd genus $2$ curves with a marked rational $3$-torsion point}

\author{Elvira Lupoian}
\address{Department of Mathematics, University College London, London, WC1H 0AY, United Kingdom}
\email{e.lupoian@ucl.ac.uk}

\author{Lazar Radi\v{c}evi\'{c}}
\address{Mathematical Institute of the Serbian Academy of Sciences and Arts, Belgrade, Serbia}
\email{lazaradicevic@gmail.com}

\begin{document}
\begin{abstract}
 In this paper we count, ordered by naive height, the genus $2$ curves over the rationals which admit a monic Weierstrass model of odd degree and whose Jacobian has a marked rational $3$-torsion point.  
\end{abstract}

\maketitle

\section{Introduction}
Let $A$ be an abelian variety defined over $\Q$. The Mordell-Weil theorem states that the group of rational points $A(\Q)$ is finitely generated. Beyond elliptic curves, our understanding of the possible rank or torsion of $A$ is limited. An interesting, and more approachable problem, is that of counting abelian varieties whose Mordell-Weil group exhibits certain properties. 

Recently, the problem of counting elliptic curves with prescribed level structure has been rather popular. Harron and Snowden \cite{harron2017counting} determined the frequency with which each torsion subgroup in Mazur's classification \cite{mazur1977modular} occurs. A number of works count elliptic curves with a marked $N$-isogeny. Boggess-Sankar \cite{boggesssankar} count elliptic curves with a rational $N$-isogeny for $N \in \{2,3,4,5,6\}$, Pizzo-Pomerance-Voight \cite{pizzo2020counting} count elliptic curves with a rational isogeny of degree $3$, Molnar-Voight \cite{molnar2023counting} count those with a $7$-isogeny and Arrango-Han-Padurariu those with a $5$-isogeny. Variations of this problem have been considered for other level structures and other fields, for instance in \cite{pomeranceshaefer}, \cite{bruin2022counting}, \cite{gajovic2025density}, \cite{cullinan2022probabilistic}, \cite{couvillion2025counting}. 

Very often the problem of counting level-structures reduces to the problem of counting points of bounded height on weighted projective stacks. A number of recent results focus on studying this problem. Phillips \cite{phillips2026points} gives asymptotics for the number of rational points in the domain of a morphism of weighted projective stacks, and uses this to count elliptic curves with certain level structures. Asymptotics for points on the modular stack of elliptic curves with specified cyclic isogen are studied by Darda and Han \cite{darda2026stacky}. 

Beyond the elliptic curve case, Genao, Phillips, Saia, Santens, and Yin \cite{genao2025counting}  count principally polarized abelian surfaces with quaternionic multiplication, or potential quaternionic multiplication, by reducing to counting points on certain genus-zero Shimura curves and their Atkin--Lehner quotients.  Chan, Loughran, and Rome \cite{chan2026thin} show that almost all odd hyperelliptic curves have Jacobians admitting no prescribed level-$G$ structure, under mild hypotheses on the level structure.  To the best of our knowledge, the present paper is the first to determine an exact asymptotic formula for the number of curves with prescribed level structure in a setting where the moduli variety is not a curve.

\subsection{Main Result}
In this paper we study genus $2$ curves with a marked rational $3$-torsion point on the Jacobian. The setup of our problems is as follows.

For a square-free polynomial 
\[
        f(x)=x^5+a_3x^3+a_2x^2+a_1x+a_0\in \Z[x]
\]
we write $C_f/\Q$ for the genus $2$ curve with a Weierstrass model $y^2 = f(x)$ and $J_f$ for the Jacobian of $C_f$.  
We define the \textit{height of} $f$ to be 
\[
        H(f)=\max\{|a_3|^{1/4}, |a_2|^{1/6}, |a_1|^{1/8}, |a_0|^{1/10}\}.
\]
We say that $f$ is \textit{minimal} if there is no prime $p$ such that $p^{10 - 2i} \vert a_i$ for all $0 \le i \le 3$. 

For $X \in \R_{\ge 0}$, we write $\calW^{\rm mark}(X)$ for the set of pairs $(f,T)$, where
\begin{enumerate}
\item $f=x^5 + a_3x^3+a_2x^2+a_1x+a_0\in\Z[x]$ is square-free and minimal;
\item $H(f)\le X$;
\item $0\ne T\in J_f(\Q)[3]$.
\end{enumerate}
Our main result is the following. 
\begin{theorem}\label{thm:main}
There exists a positive, effectively computable constant $c$ such that
\[
        \#\calW^{\rm mark}(X)=cX^{10}+o(X^{10}).
\]
More precisely,
\begin{equation}\label{eq:main-constant}
        c=\Vol(\calR(1))\,\beta_{2,3}\prod_{p\ge5}(1-p^{-10}),
\end{equation}
where $\calR(1)$ is a compact subset of $\R^4$ defined in Section
\ref{sec:height}, and $\beta_{2,3}>0$ is the effectively computable local density at the primes $2$ and $3$, defined in Section \ref{sec:count}.
\end{theorem}

A degree 5 monic integral equation as above determines a genus $2$ curve $C/\Q$ with a marked rational Weierstrass point at infinity. Every pair $(C,P)$ consisting of a genus two curve $C/\Q$ and a rational Weierstrass point $P$ can be defined by such an equation, and two such equations $y^2=x^5 + a_3x^3+a_2x^2+a_1x+a_0$ and $y^2=x^5 + a'_3x^3+a'_2x^2+a'_1x+a'_0$ define isomorphic curves if and only if there exists  $u\in \Q^\times$ such that
$a_i'=u^{10 -2i}a_i$ for $0 \le i \le 3$. 
It follows that any pair $(C,P)$ can be represented by a unique minimal integral $f$, and so we can view Theorem \ref{thm:main} as counting triples $(C,P,T)$ where $T \in \Jac C[3]$ is a non-trivial 3-torsion point. The extra data of a Weierstrass point rigidifies the moduli problem and allows us to introduce a simple height function.

\subsection{Methodology}
 The main ingredient in our proof is a parametrisation of $(f, T) \in \mathcal{W}^{\rm mark}(X)$ which is well suited to classical counting methods using the geometry of numbers. We prove the following.

\begin{proposition}
There exist weighted homogeneous polynomials 
$$\Theta_0(A,B,J,E), \Theta_1(A,B,J,E), \Theta_2(A,B,J,E), \Theta_3(A,B,J,E) \in \Z[A,B,J,E]$$
 with no common zeros except $(0,0,0,0)$, which parametrise the set of pairs $(f,T)$, where $f$ is a monic, square-free degree $5$ polynomial and $T$ is a marked $3$-torsion point on the Jacobian of the genus $2$ curve defined by $y^2 = f(x)$. 
\end{proposition} 

For the explicit definition of $\Theta_0,\ldots,\Theta_3$ see \eqref{eq:theta3}--\eqref{eq:theta0}, and for the precise statement see Proposition \ref{param}, where we show that for any $(f, T)$ as above,  
$$ f(x) = F_P(x) = x^5 + \frac{\Theta_3(P)}{729}x^3 + \frac{\Theta_2(P)}{729}x^2 + \frac{\Theta_1(P)}{729} x + \frac{\Theta_0(P)
}{729}$$
for some $P = (A,B, J,E)$ defined in terms of the coefficients of $f$. The marked $3$-torsion point $T$ is represented by a divisor cut out by a function whose coefficients  are determined by $P$. 

To prove our result, we derive necessary and sufficient conditions on $P = (A,B,J,E)$ to ensure that  $(f, T) = (F_P, T_P) \in \mathcal{W}^{\rm mark}(X)$. The requirements that $f(x) \in \Z[x]$ and that $f(x)$ is minimal correspond to $F_P(x)$ being $p$-integral and $p$-minimal for all primes $p$ (see \S \ref{sec:local} for definitions). We derive congruence conditions on $(A,B,J,E)$ characterising these properties. This essentially reduces the problem to counting rational points of bounded height on the weighted projective space $\PP(1,2,3,4)$, which we do using the geometry of numbers.

Finding the parametrisation $\Theta_0,\ldots,\Theta_3$ was the main challenge of this paper.  To explain the issue, it is helpful to first consider the analogous problem of counting elliptic curves with a marked rational $3$-torsion point treated in \cite{harron2017counting}.  We consider Weierstrass equations
\[
        y^2=x^3+Ax+B,
\]
ordered by the height function \(H(A,B)=\max(|A|^{1/4},|B|^{1/6})\), and the problem of counting minimal equations with a marked rational 3-torsion point. Every elliptic curve over \(\Q\) with a marked rational \(3\)-torsion point \(T\) can be put into Tate normal form
\[
        y^2+uxy+vy=x^3,
\]
with \(T=(0,0)\). Conversely, any such pair $(u,v) \in \Q^2$ with $\Delta(u,v)=(u^3-27v)v^3 \ne 0$ defines an elliptic curve with a marked 3-torsion point $(0,0)$. This is the universal elliptic curve over the affine modular curve $Y_1(3)$. To count short Weierstrass equations with a marked 3-torsion point, we transform the Tate normal form into short Weierstrass form $y^2=x^3+c_4(u,v)x+c_6(u,v)$, where $c_4$ and $c_6$ are polynomials in $u$ and $v$, that don't have any common zeros besides $(0,0)$. It follows that they define a morphism from the compactified modular curve $\PP(1,3) \to \PP(4,6)$. Together with a sieve argument, this reduces the problem to counting integer solutions to the inequality $H(c_4(u,v),c_6(u,v)) \leq X$.

Our genus \(2\) argument follows the same general strategy.  The first step is to write down a rational parametrisation of the family of genus \(2\) curves with a marked rational \(3\)-torsion point. This shows the modular threefold, which plays the role of the affine modular curve \(Y\) in the elliptic curve setting, is birational to the affine \(3\)-space.  The tricky part is the analogue of passing from \(Y\) to \(X\). To reduce the problem to counting lattice points in a bounded region, we need to compactify the threefold $Y$ into a weighted projective space $X$ in such a way that the rational map from $Y$ to $\PP(4,6,8,10)$ extends to a morphism $X \to \PP(4,6,8,10)$. We have managed to find such a compactification, and we show that  we can take $X=\PP(1,2,3,4)$. In the modular curve case, any rational map from a smooth projective curve extends to a morphism, and so this problem is easy to solve, but this is no longer true for higher dimensional varieties. In our case, naively homogenising the formulas will only define a rational map on the compactification, not a morphism.

Our construction is somewhat ad-hoc, and it would be interesting to understand it better, as we expect similar problems to arise when counting curves with more general level structures. For example, one can consider the family of curves \((C,P)\), where \(C\) is a hyperelliptic curve of genus \(g\) and \(P\) is a rational Weierstrass point.  In this setting, the analogue of Proposition \ref{prop:usualparam} holds, and the affine parameter space is easy to describe; the more delicate issue is to find, if it exists, a weighted projective compactification on which the map to the underlying Weierstrass model is base-point-free, and which therefore reduces the problem to counting lattice points in a compact weighted box.

\subsection{Outline}
This paper is organised as follows.  In Section \ref{sec:parametrisation} we construct the weighted homogeneous parametrisation of monic degree $5$ Weierstrass equations with a marked $3$-torsion point.  In Section \ref{sec:height} we use this parametrisation to define a compact subset in $\R^4$ in which we count lattice points.   In Section \ref{sec:local} we determine the congruence conditions which characterise lattice points that define minimal Weierstrass equations.  In Section \ref{sec:count} we combine these ingredients in a sieve and prove Theorem \ref{thm:main}.  The appendix contains \texttt{MAGMA} code used to carry out the Groebner basis computations needed in Section \ref{sec:height} and Section \ref{sec:local}.

\subsection*{Acknowledgments} We thank Netan Dogra for organising the seminar at which our collaboration began. We thank Ratko Darda and John Voight for helpful conversations. 
The first author is supported by the EPSRC Doctoral Prize fellowship EP/W524335/1 and the extension UKRI3030. The second author was supported by the Mathematical Institute
of Serbian Academy of Sciences and by the UKRI grant MR/T041609/2 during his stay at King’s College
London.

\section{The parametrisation}\label{sec:parametrisation}

We begin by recalling the elementary affine description of odd, monic, genus $2$ curves with a marked $3$-torsion point.  A similar description appears  in \cite[Lemma 3]{bruinflynntesta}.

\begin{proposition}\label{prop:usualparam}
Let $K$ be a field of characteristic different from $2$, and let
\[
        f(x)=x^5+a_3x^3+a_2x^2+a_1x+a_0\in K[x]
\]
be square-free.  Put $C_f:y^2=f(x)$.  Non-zero points
$T\in\Jac(C_f)(K)[3]$ are in bijection with triples $(A,G,H)$ where 
$A\in K^\times$, $G\in K[x]$ is monic of degree $3$, $H\in K[x]$ is monic of
 degree $2$, and
\begin{equation}\label{eq:usual-param}
        G(x)^2-A^2f(x)=H(x)^3.
\end{equation}
The $3$-torsion point corresponding to $(A,G,H)$ is $[D-2\infty]$, where $D$ is the divisor cut out by
\[
        H(x)=0,\qquad Ay=G(x).
\]
\end{proposition}

\begin{proof}
Let $0\ne T\in\Jac(C_f)(K)[3]$.  Write $T=[D-2\infty]$, where $D$ is the
reduced effective divisor of degree $2$ representing the class, allowing $\infty$
to occur in $D$.  The divisor $D$ is affine: otherwise $T=[P-\infty]$ for some
point $P$, and the relation
$3T=0$ would give a function in $L(3\infty)=\langle 1,x\rangle$ with a triple
zero at $P$, forcing $P$ to be a Weierstrass point and hence $T=0$.

Since $3T=0$, there is a function $\varphi$ with
\[
        \divisor(\varphi)=3D-6\infty.
\]
Thus $\varphi\in L(6\infty)=\langle 1,x,x^2,x^3,y\rangle$.  After scaling,
we may write
\[
        \varphi=Ay-G(x),
\]
where $G$ is monic of degree $3$.  We must have $A\ne0$; otherwise the zero
divisor of $\varphi$ is invariant under the hyperelliptic involution, which
would imply $D\sim 2\infty$ and hence $T=0$. Let \(H\) be the monic quadratic polynomial whose roots are the \(x\)-coordinates of the points in support of \(D\) (if $D=2P$ with $P=(x_P,y_P)$, we take $H=(x-x_P)^2$) then
\[
        (Ay-G)(Ay+G)=A^2f-G^2=-H^3,
\]
and therefore \eqref{eq:usual-param} holds.

Conversely, suppose that $(A,G,H)$ satisfies \eqref{eq:usual-param}.  Let $D$
be the divisor cut out by $H(x)=0$ and $Ay=G(x)$.  Since $f$ is square-free,
$G$ and $H$ are coprime.  The identity
\[
        (Ay-G)(Ay+G)=-H^3
\]
then shows that $Ay-G$ vanishes along $D$ with multiplicity $3$, and since
$G$ has a pole of order $6$ at infinity while $Ay$ has a pole of order $5$, we
obtain
\[
        \divisor(Ay-G)=3D-6\infty.
\]
Thus $[D-2\infty]$ is a $3$-torsion point.  It is non-zero. Indeed, if
$D\sim2\infty$, then $D$ is a fibre of the hyperelliptic map. Since $D$ is also cut out by $H(x)=0$ and $Ay=G(x)$, this would force $G$ and $H$ to have a common zero, and hence would force $f$ to have a double zero.
\end{proof}

Starting from polynomials $G$ and $H$ and a scalar $A$, if we define $f=(G^2-H^3)/A^2$, we obtain a genus two curve with a marked rational 3-torsion point. We now give a parametrisation of pairs $(G,H)$ for which the polynomial $f$ is monic and of degree 5 and which is suitable for lattice point counting. We consider the weighted projective space $\PP(1,2,3,4)$, with weighted coordinates $(A,B,J,E)$:
\[
        \wt(A)=1,\qquad \wt(B)=2,\qquad \wt(J)=3,\qquad \wt(E)=4.
\]
Set $P = (A,B,J,E)$ and $u=x+B/3$. We define 
\begin{align}
G_P(x) &= u^3+AJu+\frac{A^2E}{3}, \label{eq:G-u}\\
H_P(x) &= u^2-\frac{A^2}{3}u+
        \frac{-A^4+5A^2B+6AJ}{9}. \label{eq:H-u}
\end{align}
For $A\ne0$ set
\begin{equation}\label{eq:F-def}
        F_P(x)=\frac{G_P(x)^2-H_P(x)^3}{A^2}.
\end{equation}
We write
\begin{equation}\label{eq:F-theta}
        729F_P(x)=729x^5+
        \Theta_3(P)x^3+\Theta_2(P)x^2+\Theta_1(P)x+\Theta_0(P),
\end{equation}
where
\begin{align}
\Theta_3 &= -27(5A^4-30A^2B-36AJ+30B^2-18E), \label{eq:theta3}\\
\Theta_2 &= 27(6A^3J+5A^2B^2-24ABJ-20B^3+18BE-9J^2), \label{eq:theta2}\\
\Theta_1 &= 9(A^8-10A^6B-12A^5J+30A^4B^2+72A^3BJ-20A^2B^3 \notag \\
&\hspace{1.0cm}+36A^2J^2-84AB^2J+54AEJ-15B^4+18B^2E-18BJ^2), \label{eq:theta1}\\
\Theta_0 &= A^{10}-12A^8B-18A^7J+45A^6B^2+144A^5BJ-40A^4B^3 \notag \\
&\hspace{1.0cm}+108A^4J^2-252A^3B^2J-45A^2B^4-432A^2BJ^2+81A^2E^2 \notag \\
&\hspace{1.0cm}-144AB^3J+162ABEJ-216AJ^3-12B^5+18B^3E-27B^2J^2 .\label{eq:theta0}
\end{align}
The polynomials $\Theta_3,\Theta_2,\Theta_1,\Theta_0$ are weighted homogeneous
of weights $4,6,8,10$, respectively.  Although \eqref{eq:F-def} is written as
a quotient for $A\ne0$, the coefficient formula \eqref{eq:F-theta} defines
$F_P$ for every $P\in K^4$ when $729$ is invertible in $K$.

If $A\ne0$ and $F_P$ is square-free, then $C_P:y^2=F_P(x)$ is a smooth genus
$2$ curve.  Let $D_P$ be the effective divisor of degree $2$ cut out by
\[
        H_P(x)=0,\qquad Ay=G_P(x),
\]
and put
\[
        T_P=[D_P-2\infty]\in\Jac(C_P).
\]

\begin{proposition}\label{param}\label{prop:field-param}
Let $K$ be a field of characteristic different from $2$ and $3$.  Let
\[
        C_f:y^2=f(x)=x^5+a_3x^3+a_2x^2+a_1x+a_0
\]
be smooth over $K$, and let $0\ne T\in\Jac(C_f)(K)[3]$.  Then there is a
unique tuple $P=(A,B,J,E)\in K^4$, with $A\ne0$, such that
\[
        f=F_P,\qquad T=T_P.
\]
Conversely, if $A\ne0$ and $F_P$ is square-free, then $T_P$ is a non-zero point
of order $3$ on $\Jac(C_P)$.
\end{proposition}

\begin{proof}
By Proposition \ref{prop:usualparam}, the pair $(f,T)$ determines a unique
triple $(A,G,H)$ with
\[
        G(x)^2-A^2f(x)=H(x)^3,
\]
where $A\ne0$, $G$ is monic cubic, and $H$ is monic quadratic.  Write
\[
        G(x)=x^3+Bx^2+g_1x+g_0.
\]
With $u=x+B/3$, we have
\[
        G(x)=u^3+\alpha u+\beta,
\]
where
\[
        \alpha=g_1-\frac{B^2}{3},
        \qquad
        \beta=g_0-\frac{Bg_1}{3}+\frac{2B^3}{27}.
\]
Define
\[
        J=\frac{\alpha}{A},\qquad E=\frac{3\beta}{A^2}.
\]
Then $G=G_P$.  Comparing the coefficients of $x^5$ and $x^4$ in
$G^2-A^2f=H^3$ gives
\[
        H(x)=u^2-\frac{A^2}{3}u+
        \frac{-A^4+5A^2B+6AJ}{9}=H_P(x).
\]
Hence $f=F_P$, and the marked torsion point is exactly $T_P$.  This proves
existence.  The same formulae recover $A,B,J,E$ from the unique triple
$(A,G,H)$, so the tuple $P$ is unique.

Conversely, if $A\ne0$ and $F_P$ is square-free, then
\[
        G_P(x)^2-A^2F_P(x)=H_P(x)^3.
\]
Proposition \ref{prop:usualparam} therefore gives a non-zero point
$T_P\in\Jac(C_P)(K)[3]$.
\end{proof}

\section{A compact subset of \texorpdfstring{$\R^4$}{R4}}\label{sec:height}
For $X \in \R_{\ge 0}$ we define 
\[
        \calR(X)=\{P\in\R^4:H_\Theta(P)\le X\}
\]
where for any  $P\in\R^4$,
\[
        H_\Theta(P)=H(F_P)
        =\max_{0 \le i\le 3} |\Theta_i(P)/729|^{1/(10-2i)}.
\]

\begin{lemma}\label{lem:no-cusp}
The set $\calR(1)$ is compact.
\end{lemma}

\begin{proof}
We first prove that $(0,0,0,0)$ is the only common zero of $\Theta_3,\Theta_2,\Theta_1,\Theta_0$ over $\C$. Let $(A,B,J,E)$ be a solution of $\Theta_0=\Theta_1=\Theta_2=\Theta_3=0$.  As these polynomials are weighted homogeneous, if $A\ne0$, we may assume $A=1$.  A Groebner basis computation (see Appendix \ref{app:no-cusp}) gives
\begin{equation}\label{eq:gb-nine-use}
        1\in
        (\Theta_3(1,B,J,E),\Theta_2(1,B,J,E),
        \Theta_1(1,B,J,E),\Theta_0(1,B,J,E))
        \subset \Q[B,J,E],
\end{equation}
and so there are no solutions with $A \ne 0$. Suppose next $A=0$. The equations
$\Theta_3=\Theta_2=0$ imply
\[
        E={\frac{5}{3}}B^2,
        \qquad
        10B^3-9J^2=0.
\]
Substituting these relations into $\Theta_1=0$ gives $-5B^4=0$.  Hence
$B=J=E=0$. Thus the only common zero of the four forms is the origin.

Now put
\[
        \|P\|_w=\max\{|A|,|B|^{1/2},|J|^{1/3},|E|^{1/4}\}
\]
and  $S_w=\{P\in\R^4:\|P\|_w=1\}$.  This is a compact set on which the function $H_{\Theta}$ is strictly positive, and hence
$
        m:=\min_{P\in S_w} H_\Theta(P)>0.
$ For any non-zero $P\in\R^4$, write $\lambda=\|P\|_w$ and
$
        Q=(\lambda^{-1}A,\lambda^{-2}B,\lambda^{-3}J,\lambda^{-4}E)
$. Then $Q\in S_w$, and weighted homogeneity gives
\[
        H_\Theta(P)=\lambda H_\Theta(Q)\ge m\lambda.
\]
Thus $\calR(1)\subset\{P:\|P\|_w\le m^{-1}\}$.  Since $\calR(1)$ is closed,
it is compact.
\end{proof}

\begin{lemma}\label{lem:lattice}
Let $v+\Lambda\subset\Q^4$ be an affine lattice of full rank.  Then
\[
        \#(\calR(X)\cap(v+\Lambda))
        =\frac{\Vol(\calR(1))}{\operatorname{covol}(\Lambda)}X^{10}+O_\Lambda(X^9).
\]
If $Z\subset\A^4$ is a proper algebraic subvariety, then
\[
        \#(\calR(X)\cap(v+\Lambda)\cap Z(\Q))=O_{Z,\Lambda}(X^9).
\]
\end{lemma}

\begin{proof}
Weighted homogeneity gives
\[
        \calR(X)=\{(XA,X^2B,X^3J,X^4E):(A,B,J,E)\in\calR(1)\}.
\]
The determinant of this dilation is $X^{1+2+3+4}=X^{10}$.  Since $\calR(1)$ is
compact and semi-algebraic, Davenport's Lipschitz principle \cite{davenport1951principle} gives the lattice
point estimate.  The subvariety estimate is an easy consequence of this.
\end{proof}

\section{Local conditions}\label{sec:local}

Define 
\[
        c_i(P)=\Theta_i(P)/729 \qquad (i=0,1,2,3),
\]
so that
\[
        F_P(x)=x^5+c_3(P)x^3+c_2(P)x^2+c_1(P)x+c_0(P).
\]
For a prime $p$ we say that $F_P$ is \textit{$p$-integral} if all four
coefficients $c_i(P)$ lie in $\Z_p$. When $F_P$ is $p$-integral, we 
say that it is \textit{not $p$-minimal} if
\[
        p^4\mid c_3(P),\qquad
        p^6\mid c_2(P),\qquad
        p^8\mid c_1(P),\qquad
        p^{10}\mid c_0(P),
\]
and \textit{$p$-minimal} otherwise.  
\begin{proposition}\label{prop:bad-primes}
Let $p$ be a prime. The set $N_p$ of tuples \(P\in\Q_p^4\) for which \(F_P\) is \(p\)-integral and
\(p\)-minimal is a compact subset of $\Q_p^4$. Moreover, $N_p$ is described by a finite set of congruence conditions, and 
there exists a constant $M_p$ such that $p^{M_p} N_p \subset \Z_p^4$.  
\end{proposition}
\begin{proof}
 We first show that there is an effectively
computable integer \(M_p\) such that for any $P \in \Q_p^4$, $F_P \in \Z_p[x]$ only if $P \in p^{-M_p} \Z_p^4$. A \texttt{MAGMA} computation (see Appendix \ref{app:no-cusp}) shows that the common zero locus of the four
\(\Theta_i\) over \(\overline{\Q}\) is the origin $(0,0,0,0)$.  Hence, for each
coordinate \(X\in\{A,B,J,E\}\), some power of \(X\) belongs to the ideal
generated by  \(\Theta_i\) over \(\Q\).  We have
\begin{equation}\label{eq:nullstellensatz-coord}
        C_X X^{N_X}
        =
        R_{X,3}\Theta_3+R_{X,2}\Theta_2+R_{X,1}\Theta_1+R_{X,0}\Theta_0,
\end{equation}
with \(C_X\in\Z_{\ne0}\), \(N_X>0\), and
\(R_{X,i}\in\Z[A,B,J,E]\).  These identities are effectively computable
from a Groebner basis.

Now take \(P=(A,B,J,E)\in\Q_p^4\), and let $n$ be the smallest integer
such that 
$$
        Q=(p^nA,\ p^{2n}B,\ p^{3n}J,\ p^{4n}E) \in \Z_p^4.
$$

Since \(Q\in\Z_p^4\), the values \(R_{X,i}(Q)\) are \(p\)-adic integers.  The identity
\eqref{eq:nullstellensatz-coord}  implies that if all the values
\(\Theta_i(Q)\) had sufficiently large \(p\)-adic valuation, then each coordinate of
$Q$ would be divisible by the corresponding weighted power of $p$, but this contradicts the  the definition of $n$ . It follows that there exists a constant \(K_p\), independent of $P$, such
that, for every such \(Q\),
\begin{equation}\label{eq:theta-valuation-bound-simple}
        \min_i v_p(\Theta_i(Q))\le K_p .
\end{equation}

Suppose now that \(F_P \in \Z_p[x]\) is integral.  As each \(\Theta_i\) is
weighted homogeneous of degree $d_i=10-2i$, we have
\[
        \Theta_i(Q)=p^{nd_i}\Theta_i(P).
\]
Also \(c_i=\Theta_i/729\), so \(c_i(P)\in\Z_p\) gives
\[
        v_p(\Theta_i(P))\ge v_p(729).
\]
Thus
\[
        v_p(\Theta_i(Q))\ge nd_i+v_p(729)
        \qquad\text{for all }i.
\]
Combining this with \eqref{eq:theta-valuation-bound-simple}, and using
\(d_i\ge4\), gives
\[
        4n+v_p(729)\le K_p.
\]
Hence \(n\) is bounded above, and  we can take
\[
        M_p=\max\left\{0,\left\lceil
        \frac{K_p-v_p(729)}{4}
        \right\rceil\right\}.
\]

Write
\[
        P=(p^{-M_p}a,\ p^{-2M_p}b,\ p^{-3M_p}j,\ p^{-4M_p}e),
        \qquad (a,b,j,e)\in\Z_p^4.
\]
For each \(i\), choose \(e_{p,i}\ge0\) minimal such that
\[
        \Psi_{p,i}(a,b,j,e)
        :=
        p^{e_{p,i}}c_i(p^{-M_p}a,p^{-2M_p}b,p^{-3M_p}j,p^{-4M_p}e)
\]
belongs to \(\Z_p[a,b,j,e]\).  Then
\[
        c_i(P)\in\Z_p
        \quad\Longleftrightarrow\quad
        \Psi_{p,i}(a,b,j,e)\equiv0\pmod {p^{e_{p,i}}}.
\]
Moreover \(F_P\) is  not \(p\)-minimal exactly when
\[
\begin{aligned}
        \Psi_{p,3}(a,b,j,e)&\equiv0\pmod {p^{e_{p,3}+4}},&
        \Psi_{p,2}(a,b,j,e)&\equiv0\pmod {p^{e_{p,2}+6}},\\
        \Psi_{p,1}(a,b,j,e)&\equiv0\pmod {p^{e_{p,1}+8}},&
        \Psi_{p,0}(a,b,j,e)&\equiv0\pmod {p^{e_{p,0}+10}}.
\end{aligned}
\]
Therefore \(p\)-integrality and \(p\)-minimality are determined by the
residue class of \((a,b,j,e)\) modulo $p^{L_p}$, where $L_p=\max\{e_{p,3}+4,e_{p,2}+6,e_{p,1}+8,e_{p,0}+10\}.$ 
\end{proof}

For $p \geq 5$ these congruence conditions take a simple form.
\begin{proposition}\label{prop:good-primes}
Let $p\ge5$.  For $P\in\Z_p^4$, the polynomial $F_P$ is $p$-integral. The polynomial $F_P$ is not $p$-minimal if and only if
\begin{equation}\label{eq:good-nonminimal}
        p\mid A,
        \qquad p^2\mid B,
        \qquad p^3\mid J,
        \qquad p^4\mid E.
\end{equation}
The density of the subset of $\Z_p^4$ of points $P$  for which $F_P$ is $p$-minimal is
$1-p^{-10}$.
\end{proposition}

\begin{proof}
The ``if" direction is clear.  Assume that $P$ is integral and that  $F_P$ is not $p$-minimal.  Reducing the four $\Theta_i$ modulo $p$ shows
first that
\[
        A\equiv B\equiv J\equiv E\equiv0\pmod p.
\]
If $A\ne0$, we can assume $A=1$ by weighted homogeneity.  The Groebner basis in
Appendix \ref{app:no-cusp} shows that the corresponding ideal contains $9$ over
$\Z$, so it has no zero modulo $p\ge5$.  If $A=0$, then the calculation in the
proof of Lemma \ref{lem:no-cusp} works modulo $p\ge7$; for $p=5$ the first
three equations give $E=0$, $J=0$, and then $\Theta_0=-12B^5$ gives $B=0$.

Write $A=pa$, $B=pb$, $J=pj$, and $E=pe$.  From $p^4\mid\Theta_3$ we get
\[
        \frac{\Theta_3(pa,pb,pj,pe)}{p}
        \equiv 486e\pmod p,
\]
so $E=p^2e_1$.  From $p^6\mid\Theta_2$ we get
\[
        \frac{\Theta_2(pa,pb,pj,p^2e_1)}{p^2}
        \equiv -243j^2\pmod p,
\]
so $J=p^2j_1$.

Now substitute $A=pa$, $B=pb$, $J=p^2j_1$, and $E=p^2e_1$.  The conditions
$p^4\mid\Theta_3$ and $p^6\mid\Theta_2$ give
\[
        5b^2-3e_1\equiv0,
        \qquad -20b^3+18be_1\equiv0\pmod p.
\]
For $p\ge7$ these imply $10b^3\equiv0\pmod p$, so $B=p^2b_1$ and
$E=p^3e_2$.  Substituting $A=pa$, $B=p^2b_1$, $J=p^2j_1$, $E=p^3e_2$ into
$\Theta_2$ gives
\[
        \frac{\Theta_2(pa,p^2b_1,p^2j_1,p^3e_2)}{p^4}
        \equiv -243j_1^2\pmod p,
\]
so $J=p^3j_2$.  Finally
\[
        \frac{\Theta_3(pa,p^2b_1,p^3j_2,p^3e_2)}{p^3}
        \equiv 486e_2\pmod p,
\]
so $E=p^4e_3$.

For $p=5$, the same argument gives $A=5a$, $B=5b$, $J=25j$, and $E=25e_1$.
Using $\Theta_3$ once more gives $e_1\equiv0\pmod5$, so $E=125e$.  With
$A=5a$, $B=5b$, $J=25j$, $E=125e$, and since $F_P$ is not $5$-minimal we get  the four
congruences
\[
\begin{aligned}
2aj-2b^2+e&\equiv 0 \bmod 5,\\
2abj+2b^3+be+2j^2&\equiv 0 \bmod 5,\\
-ab^2j-2b^4+2b^2e-2bj^2&\equiv 0 \bmod 5,\\
-2b^5&\equiv 0 \bmod 5.
\end{aligned}
\]
The last congruence gives $b\equiv0$, the second gives $j\equiv0$, and the
first gives $e\equiv0$. Thus \eqref{eq:good-nonminimal} holds also for
$p=5$.

The sub-lattice defined by \eqref{eq:good-nonminimal} has index $p^{1+2+3+4}=p^{10}$.
This gives the density $1-p^{-10}$ for minimality at $p$.
\end{proof}

\section{The sieve}\label{sec:count}

Proposition \ref{prop:bad-primes}, applied at $p=2$ and $p=3$, gives a finite
set of congruence conditions on $P=(A,B,J,E)\in\Q^4$ which characterize
integrality and minimality at these two primes.  We impose these conditions on
$\Z[1/6]^4$. The subset $  \mathcal L_{2,3}^{\min}$ of $\Z[1/6]^4$ consisting of tuples which are
integral and minimal at both $2$ and $3$ is covered by a set $I$ of affine
lattices 
\[
        \mathcal L_{2,3}^{\min}
        =
        \bigsqcup_{\nu\in I} (v_\nu+\Lambda_{2,3}),
\]
contained in $\Z[1/6]^4$, with
$\Lambda_{2,3}\subset \Q^4$ a lattice of full rank.  We define
\[
        \beta_{2,3}
        =
        \frac{\# I}{\operatorname{covol}(\Lambda_{2,3})}.
\]
Equivalently, by Lemma \ref{lem:lattice},
\begin{equation}\label{eq:bad-density}
        \#\bigl(\calR(X)\cap\mathcal L_{2,3}^{\min}\bigr)
        =
        \Vol(\calR(1))\,\beta_{2,3}X^{10}+O(X^9).
\end{equation}
The number $\beta_{2,3}$ is effectively computable from the congruence
conditions of Proposition \ref{prop:bad-primes}. This computation could be carried out in practice with enough persistence, but appears to be rather tedious.

For a prime $p\ge5$ put
\[
        p\cdot P=(pA,p^2B,p^3J,p^4E).
\]
Define
\[
\begin{aligned}
\calM(X)=\{P\in\calR(X)\cap\mathcal L_{2,3}^{\min}:&\ A\ne0,
\ \disc(F_P)\ne0, \\
&\ P\notin p\Z_p \times p^2 \Z_p\times  p^3 \Z_p  \times p^4\Z_p \text{ for every prime }p\ge5\}.
\end{aligned}
\]
The last condition says that, for every prime $p \geq 5$, the four
coordinates are not simultaneously divisible by $p$ with weights $1,2,3,4$.

\begin{proposition}\label{prop:bijection}
The map
\[
        P\longmapsto (F_P,T_P)
\]
is a bijection from $\calM(X)$ to $\calW^{\rm mark}(X)$.
\end{proposition}

\begin{proof}
This follows from Proposition \ref{prop:field-param}, Proposition \ref{prop:bad-primes} and Proposition \ref{prop:good-primes}.
\end{proof}

\begin{remark}\label{rem:nonintegral}
The tuple $P$ need not be integral at the primes $2$ and $3$. For instance, 
\[
        P=\left(6,0,0,\frac{9}{2}\right)
\]
gives
\[
        F_P(x)=x^5-237x^3+20736x+83025,
\]
which is integral and minimal.
\end{remark}
Set
\[
\begin{aligned}
\calG(X)=\{P\in\calR(X)\cap\mathcal L_{2,3}^{\min}:&
P\notin p\cdot\Z[1/6]^4\text{ for every prime }p\ge5\}.
\end{aligned}
\]
Then
\begin{equation}\label{eq:groomed-count}
        \calM(X)=\calG(X)\setminus\bigl(\{A=0\}\cup\{\disc(F_P)=0\}\bigr),
\end{equation}
and Proposition \ref{prop:bijection} identifies $\calM(X)$ with
$\calW^{\rm mark}(X)$.

\begin{proposition}\label{prop:main-count}
As \(X\to\infty\),
\[
        \#\calG(X)=
        \Vol(\calR(1))\,\beta_{2,3}\prod_{p\ge5}(1-p^{-10})X^{10}
        +o(X^{10}).
\]
\end{proposition}

\begin{proof}
Let \(S\) be a finite set of primes \(p\ge5\).  Since
\(\mathcal L_{2,3}^{\min}\) is a finite union of affine lattices whose
denominators are supported only at \(2\) and \(3\), reduction modulo powers of
primes in \(S\) behaves as for the standard lattice.  Proposition
\ref{prop:good-primes} and Lemma \ref{lem:lattice} give
\[
\begin{aligned}
&\#\{P\in\calR(X)\cap\mathcal L_{2,3}^{\min}:
        P\notin p\cdot\Z[1/6]^4\text{ for every }p\in S\}  \\
&\qquad =
        \Vol(\calR(1))\,\beta_{2,3}
        \prod_{p\in S}(1-p^{-10})X^{10}+O_S(X^9).
\end{aligned}
\]

It remains to pass from the finite sieve to the full sieve.  For $M\ge 5$, let
$N_M^{\rm tail}(X)$ be the number of points $P\in\calR(X)\cap
\mathcal L_{2,3}^{\min}$, with $A\ne0$, such that $F_P$ is non-minimal at some
prime $p>M$.  We will show that
\begin{equation}\label{eq:tail-sieve-bound}
        N_M^{\rm tail}(X)
        \ll
        X^{10}\sum_{p>M}p^{-10}+o(X^{10}),
\end{equation}
where the implied constant is independent of $M$ and $X$.

Points in \(\calR(X)\) satisfy
\[
        |A|\ll X,
        \quad |B|\ll X^2,
        \quad |J|\ll X^3,
        \quad |E|\ll X^4.
\]
Fix a prime $p \geq 5$. By Proposition \ref{prop:good-primes}, non-minimality at
$p$ is equivalent to
\[
        p\mid A,\qquad p^2\mid B,\qquad p^3\mid J,\qquad p^4\mid E.
\]
Therefore, in any one of the affine lattices $v_\nu+\Lambda_{2,3}$, the number
of $(A,B,J,E)$ that give integral equations that are non-minimal at $p$ is bounded by
\[
        \ll
        \left(\frac{X}{p}+1\right)
        \left(\frac{X^2}{p^2}+1\right)
        \left(\frac{X^3}{p^3}+1\right)
        \left(\frac{X^4}{p^4}+1\right),
\]
with the implied constant depending only on the finite collection of lattices.  Since
$A\ne0$, any such prime $p$ divides the non-zero integer obtained from $A$ after
clearing its denominator, which is uniformly bounded by  Proposition
\ref{prop:bad-primes}, and as 
$|A|\ll X$ on \(\calR(X)\), only primes $p\ll X$ can occur.  Summing the
preceding estimate over $M<p\ll X$ and over the finitely many affine lattices
proves \eqref{eq:tail-sieve-bound}.

The contribution from \(A=0\) is \(O(X^9)\) by Lemma \ref{lem:lattice}.  Taking
$X\to\infty$ in the finite sieve estimate and using \eqref{eq:tail-sieve-bound}
gives
\[
        \limsup_{X\to\infty}\left|\frac{\#\calG(X)}{X^{10}}
        -\Vol(\calR(1))\,\beta_{2,3}\!\!\prod_{5\le p\le M}\!(1-p^{-10})\right|
        \;\ll\;\sum_{p>M}p^{-10}.
\]
The result follows by taking \(M\to\infty\).
\end{proof}

\begin{proof}[Proof of Theorem \ref{thm:main}]
By Proposition \ref{prop:main-count}, it remains only to remove the degenerate tuples $(A,B,J,E)$ with
\(A=0\) and \(\disc(F_P)=0\). By Lemma \ref{lem:lattice}, these loci
contribute \(O(X^9)\) points. We can identify the remaining tuples with \(\calW^{\rm mark}(X)\) by  \eqref{eq:groomed-count} and Proposition \ref{prop:bijection}, and the conclusion follows.
\end{proof}
\begin{remark}
A natural question one can ask is how our result compares to counting genus $2$ curves $C$ whose Jacobian has a non-trivial  rational $3$-torsion point. This amounts to forgetting the marking of the torsion point $T$. Let $\mathcal{W}^{\text{unmarked}}(X)$ denote the set of such curves of height at most $X$. Since $\#\Jac(C)(\bar{\Q})[3]=81$, we have $\#\mathcal{W}^{\text{unmarked}}(X)=O(X^{10})$. In fact, we can pin down the leading constant exactly, and we have
\[
    \#\mathcal{W}^{\text{unmarked}}(X)=\frac{c}{2}X^{10}+o(X^{10}).
\]
We give a brief sketch of the argument. As the Weil pairing is Galois equivariant, the group $\Jac(C)(\Q)[3]$ is isotropic for the Weil pairing, and hence is of order at most $9$. Curves with a $\Q$-rational $3$-torsion subgroup of order $9$ correspond to a thin subset $M$ of $\PP(1,2,3,4)(\Q)$. To see this, we consider the moduli space of tuples $(C,P,T,K)$ where $C$ is a genus 2 curve, $P$ is a Weierstrass point, $T$ is a 3-torsion point on $\Jac C$, and  $K$ is a maximal isotropic subgroup of $\Jac C[3]$ such that $\langle T \rangle \leq K$, and the map $(C,P,T,K) \to (C,P,T)$ which forgets the subgroup $K$. Then $M$ is contained in the image of this map, and it follows that $M$ is a thin set in the sense of Serre. Hence, by \cite[Theorem 1.1]{chan2026thin} applied to $\PP(1,2,3,4)$, there are $o(X^{10})$ of them. Outside this negligible subset, the rational fibers of the map that forgets the marking $(C,P,T)\mapsto (C,P)$
are of size two, corresponding to the points $T$ and $-T$, and so the asymptotic count is half that of the count of curves with a marked $3$-torsion point.

It would be interesting to determine exactly the asymptotic count of genus $2$ curves $C$ with $\#\Jac(C)(\Q)[3]=9$. We expect that this can be achieved using the results of \cite{bruinflynntesta} and the methods of this paper.
\end{remark}
\bibliographystyle{plain}
 \bibliography{references}

\appendix

\section{A Groebner basis computation}\label{app:no-cusp}
The only computer calculation used in the proof is the following integer
Groebner basis computation, needed in the proof of Lemma \ref{lem:no-cusp} and in
the proof of Proposition \ref{prop:good-primes}. Here is the Magma code.  

\begin{verbatim}
Z := Integers();
R<B,J,E> := PolynomialRing(Z, 3, "lex");
A := R!1;

T3 := -27*(5*A^4 - 30*A^2*B - 36*A*J + 30*B^2 - 18*E);
T2 :=  27*(6*A^3*J + 5*A^2*B^2 - 24*A*B*J - 20*B^3
           + 18*B*E - 9*J^2);
T1 :=   9*(A^8 - 10*A^6*B - 12*A^5*J + 30*A^4*B^2
           + 72*A^3*B*J - 20*A^2*B^3 + 36*A^2*J^2
           - 84*A*B^2*J + 54*A*E*J - 15*B^4
           + 18*B^2*E - 18*B*J^2);
T0 :=      A^10 - 12*A^8*B - 18*A^7*J + 45*A^6*B^2
           + 144*A^5*B*J - 40*A^4*B^3 + 108*A^4*J^2
           - 252*A^3*B^2*J - 45*A^2*B^4 - 432*A^2*B*J^2
           + 81*A^2*E^2 - 144*A*B^3*J + 162*A*B*E*J
           - 216*A*J^3 - 12*B^5 + 18*B^3*E - 27*B^2*J^2;

I := ideal< R | T3, T2, T1, T0 >;
GB := GroebnerBasis(I);
GB;

assert &or[ g eq R!9 or g eq R!-9 : g in GB ];
\end{verbatim}

\end{document}